\documentclass[a4paper,12pt]{amsart}   
\def\myfnt{\ifx\protect\@typeset@protect\expandafter\footnote\else\expandafter\@gobble\fi}
\makeatother
\usepackage{amsfonts}
 \usepackage{amssymb}
 \usepackage{amsmath,amsxtra,amsthm}

\usepackage[usenames]{color}
\usepackage[mathscr]{eucal}

 \usepackage{dsfont}

\usepackage{hyperref}

\usepackage{wrapfig}

\usepackage[all]{xy}

\usepackage{shuffle}
\usepackage{scalerel}[2016/12/29]

\newtheorem{theorem}{Theorem}
\newtheorem{deff}{Definition}

\newtheorem{proposition}{Proposition}
\newtheorem{example}{Example}
\newtheorem{lemma}{Lemma}

\newtheorem{prop}{Proposition}
\newtheorem{rem}{Remark}

\newcommand{\mto}{\mapsto}
\newcommand{\bqa}{\begin{eqnarray}}
\newcommand\eqa {\end{eqnarray}}
\newcommand{\beq}{\begin{eqnarray}}
\newcommand{\beqn}{\begin{eqnarray}\nonumber}
\newcommand{\eeq}{\end{eqnarray}}
\newcommand{\be}{\begin{array}}
\newcommand{\ee}{\end{array}}

 \newcommand{\pt}{\partial}

   \newcommand\vf\varphi

 \newcommand{\cM}{{\mathcal M}}

 \newcommand{\cI}{{\mathcal I}}

 \newcommand{\cF}{{\mathcal{F}}}

\newcommand{\cE}{{\mathcal E}}

 \newcommand{\cv}{\mathpzc{v}}

 \newcommand{\Z}{{\mathbb Z}}

 \newcommand{\N}{{\mathbb N}}

\newcommand{\ua}{{\underline{a}}}
\newcommand{\ub}{{\underline{b}}}

   \def\e{\epsilon}

\DeclareFontFamily{OT1}{pzc}{}
\DeclareFontShape{OT1}{pzc}{m}{it}{<-> s * [1.15] pzcmi7t}{}
\DeclareMathAlphabet{\mathpzc}{OT1}{pzc}{m}{it}

\usepackage{tikz}
 \usetikzlibrary{math}
 \usetikzlibrary{arrows.meta, calc}

 \usepackage{pgfplots}
\pgfplotsset{compat=1.18}
    \usepgfplotslibrary{patchplots}

    \newcommand{\cg}{\mathpzc{g}}

\newcommand{\p}{\mathrm{pr}}
\newcommand{\wpr}{\widehat{\p}}
\newcommand{\wE}{\widehat{E}}
\newcommand{\wF}{\widehat{F}}
\newcommand{\Exp}{\mathrm{Exp}}
\newcommand{\vx}{{\vartriangle\!\!x}}

   \usepackage{bm}
   \usepackage{mathtools,graphicx}
   \usepackage{sfmath}

\begin{document}
    
\bibliographystyle{amsplain}

\title{The functor between two categories of $\Z-$graded manifolds}

\author{Martha Valentina Guarin Escudero$^{1,2}$}

\author{Alexei Kotov$^1$}
\address{1. Faculty of Science, University of Hradec Kralove, Rokitanskeho 62, Hradec Kralove
50003, Czech Republic}

\address{2. Mathematical Institute of Charles University,
Sokolovská 49/83, Prague-8 186 75, Czech Republic }

\email{val.guarine@matfyz.cuni.cz, oleksii.kotov@uhk.cz}


\begin{abstract} 
This paper examines \(\mathbb{Z}\)-graded manifolds as semiformal homogeneity structures, comparing two polynomial filtrations from their local models. In finite dimensions, these are componentwise equivalent, yielding isomorphic graded completions; generally, one induces a finer topology. By the Batchelor-Gawedzki-type theorem (Kotov--Salnikov), every \(\mathbb{Z}\)-graded manifold over base M is noncanonically isomorphic to one associated with its canonical \(\mathbb{Z}\)-graded bundle (Batchelor-Gawedzki bundle). In finite dimensions, this is the formal neighborhood of the zero section with the induced homogeneity structure. Kotov-Salnikov's graded Borel lemma extends weight-k functions from the formal neighborhood to smooth ones of the same weight. Here, this generalizes to a Borel--Whitney theorem: homogeneity morphisms of formal neighborhoods lift to smooth homogeneity maps between Batchelor-Gawedzki bundles. Categorically, let \(\mathsf{B}_{\mathbb{Z}}\) be the category of finite-dimensional \(\mathbb{Z}\)-graded vector bundles with homogeneity morphisms, and \(\mathsf{Man}_{\mathbb{Z}}\) the category of finite-dimensional \(\mathbb{Z}\)-graded manifolds. The functor \(\mathsf{F}\colon \mathsf{B}_{\mathbb{Z}} \to \mathsf{Man}_{\mathbb{Z}}\) sends bundles to formal neighborhoods of their zero sections. The graded Batchelor-Gawedzki and Borel-Whitney theorems imply \(\mathsf{F}\) is full and surjective on objects.
\end{abstract}

\maketitle

\renewcommand{\theequation}{\thesection.\arabic{equation}}
\section*{Introduction}\label{sec:introduction}

\noindent Graded manifolds provide the natural geometric language for theories where physical fields and their symmetries possess different degrees of "ghost number" or "fermion number." In the BV-BRST formalism, the configuration space is extended to a graded manifold (a shifted cotangent bundle) where the classical master equation is expressed via a homological vector field, allowing for the consistent quantization of gauge theories with open algebras. Similarly, in supersymmetry, graded manifolds underpin the concept of superspace, where the mixing of bosonic and fermionic coordinates explains the pairing of particles with different spins and facilitates the construction of invariant Lagrangians. A careful understanding of the underlying mathematics of graded manifolds is essential, as it ensures that the transition from classical geometry to the algebraic structures of physics is both rigorous and consistent.

\smallskip\noindent The concept underlying vector spaces with grading determined by an abelian group or monoid exhibits independence from the particular selection of the group or monoid.
 Consider a vector space defined over a field $k$ with grading imposed by an abelian monoid $\Gamma$. This structure manifests as a direct sum decomposition:
\[
V = \bigoplus_{\gamma \in \Gamma} V_\gamma,
\]
where each $V_\gamma$ represents the homogeneous subspace associated with degree (or weight) $\gamma$. Linear mappings $\phi: V \to W$ between such graded spaces are called morphisms if they are \emph{weight-preserving}, that is, if $\phi(V_\gamma) \subseteq W_\gamma$ for every $\gamma \in \Gamma$. A linear map is said to act with a \emph{constant weight shift}, or to be of pure weight, if there exists a fixed $\delta \in \Gamma$ such that $\phi(V_\gamma) \subseteq W_{\gamma + \delta}$ for all $\gamma \in \Gamma$.
The category of $\Gamma$-graded vector spaces and morphisms between them carries a natural tensor (monoidal) structure, defined by the usual tensor product of vector spaces together with the grading rule
\[
(V \otimes W)_\gamma = \bigoplus_{\alpha + \beta = \gamma} V_\alpha \otimes W_\beta.
\]
This monoidal structure facilitates the formation of \emph{associative graded algebras}, interpreted as monoids within this category. Introducing a compatible braiding $\sigma_{V,W}: V \otimes W \to W \otimes V$ transforms it into a \emph{braided tensor category}; if this braiding squares to the identity morphism (i.e., $\sigma_{W,V} \circ \sigma_{V,W} = \mathrm{id}_{V \otimes W}$), it becomes \emph{symmetric}, yielding \emph{commutative associative graded algebras}. In geometric or physical applications, this braiding often follows a \emph{parity-based sign convention}.

\smallskip\noindent Translating graded algebra into the realm of manifolds or generalized spaces encounters substantial obstacles. No single canonical methodology exists for defining ``graded manifolds'' universally, owing to divergent choices in coordinate systems and permissible function classes across the literature. Consequently, structures deemed ``analogous'' may correspond to fundamentally distinct realizations, influenced by whether the context is smooth, analytic, or formal, and by the specific grading monoid employed.

\smallskip\noindent The $\mathbb{Z}_2$-graded case stands out due to its extensive development in theory. The tensor symmetry follows the sign rule
\[
v \otimes w \mapsto (-1)^{p(v)p(w)} w \otimes v,
\]
where $p: \mathbb{Z}/2\mathbb{Z}$ denotes the parity of homogeneous elements. Formally, a \emph{supermanifold} \cite{Berezin1966,Leites1977} is a \emph{locally ringed space} over a smooth base, where the structure sheaf  is locally free as a sheaf of $\mathbb{Z}_2$-graded commutative algebras, generated locally by smooth functions of even (weight 0) commuting variables and \emph{Grassmann polynomials} in odd (weight 1) anticommuting variables.
Supermanifold morphisms are morphisms of $\mathbb{Z}_2$-graded commutative ringed spaces covering smooth base maps.


\smallskip\noindent The Batchelor \cite{Batchelor1979} (or Gawedzki \cite{Gawedzki1977}) theorem  guarantees that every real smooth supermanifold $\mathcal{M}$ above base $M$ admits a canonical vector bundle $E \to M$ rendering its structure sheaf non-canonically isomorphic to $\Gamma(M, \Lambda^\bullet E^*)$. Equivalently, $\mathcal{M} \cong \Pi E$, where $\Pi E$ bears the sheaf of exterior forms on $E$.
This bundle $E$ arises uniquely (up to isomorphism) as the normal bundle to the canonical embedding of $M$ in $\mathcal{M}$. Such a classification eludes complex-analytic supermanifolds, permitting \emph{split} (globally isomorphic to $\Pi E$) and \emph{non-split} forms.

\smallskip\noindent Within the ringed space paradigm, $\mathbb{N}$-graded manifolds employ sheaves locally blending smooth functions of degree-$0$ variables with polynomials in positive-degree variables. $\mathbb{Z}$-graded manifolds are (component-wise) locally ringed spaces locally modelled by the completion with respect to the canonical filtration\footnote{The idea of the filtration was inspired by the work of Felder and Kazhdan \cite{felder-kazhdan}, where it arose in the context of affine varieties.} of the polynomial algebra in variables of non-zero weight, with coefficients smooth functions of weight-$0$ variables. This notion was formalized by Kotov--Salnikov~\cite{KS2024}, with a parallel formulation by Vysok\'y \cite{vysoky}. Real smooth settings admit Batchelor-Gawedzki analogues: Roytenberg's for $\mathbb{N}$-grading, Kotov--Salnikov's for $\mathbb{Z}$-grading. Complex-analytic counterparts fail, necessitating distinctions between split and non-split types.

\smallskip\noindent
 Alternatively, Grabowski and Rotkiewicz \cite{grab-hom} recast $\mathbb{N}$-graded manifolds via $\mathbb{R}_{\geq 0}$-actions on manifolds, with infinitesimal generator termed the \emph{Euler vector field}. Locally, this induces rescalings of homogeneous (with respect to the Euler field) variables by nonnegative weights. The action's fixed points form a smooth submanifold serving as base for a canonical graded bundle, whose total space is the original graded manifold. This finds significant applications in differential geometry; for instance, it provides a suitable formalism for objects in the VB category.

\smallskip\noindent Another global approach to $\Z$-graded manifolds (due to \cite{Voronov2002}; 
see also applications in \cite{GrabowskaGrabowski2024, GrabowskaGrabowski2025, Grabowskaetal2025, KS-HCP} and \cite{BruceGrabowskiRotkiewicz2016, BruceGrabowskaGrabowski2018})
views them as smooth (super)manifolds equipped with an even vector field of the form of an Euler field (allowing both positive and negative weights) near its zero locus.
The special form of this vector field guarantees that the zero 
locus is a smooth submanifold. These objects, called 
\emph{homogeneity manifolds} or \emph{manifolds with homogeneity 
structure}, form a category whose morphisms are smooth maps 
preserving the Euler fields. 
The zero-locus formal neighborhoods of homogeneity manifolds are 
$\Z$-graded manifolds in the original sense (termed \emph{semiformal} 
by Kotov--Salnikov \cite{KS2024}). By the Batchelor theorem~\cite{KS2024}, every semiformal manifold is the formal neighborhood of the zero locus of some homogeneity structure: specifically, the one on its Batchelor bundle $E$. Furthermore, every weight-$k$ semiformal function $f$ on $E$ 
extends to a smooth function $\widetilde{f}$ on $E$ of the same 
weight, matching the normal jet of $f$ at the base $M$. 
This \emph{graded Borel theorem} \cite{borel1901, bourbaki1967} extends the classical Borel 
lemma, which realizes formal power series as smooth jets, even 
when divergent and non-analytic, while requiring preservation 
of the weight grading.

\smallskip\noindent 
The paper is organized as follows. Section~\ref{sec:semiformal} compares two filtrations whose completions of the polynomial algebra yield different local models of $\mathbb{Z}$-graded manifolds. It is shown that in the finite-dimensional case these filtrations are equivalent on homogeneous components of each weight, producing isomorphic graded completions.
 In the case of finite graded dimension, inclusion holds only one way: one filtration induces a finer filtered topology. Section~\ref{sec:formal_homogeneity} explains, through the local model example, in what sense $\mathbb{Z}$-graded manifolds should be regarded as semiformal homogeneity structures. Section~\ref{sec:global_Borel} proves a Borel--Whitney-type theorem for morphisms of finite-dimensional $\mathbb{Z}$-graded manifolds. The concluding Section~\ref{sec:beyond}, \emph{Beyond the Present Scope: Homogeneity structures}, outlines several examples and perspectives that motivated this study and indicate possible future research.
 Appendix~\ref{sec:app} reviews filtrations and completions, proving a lemma used in Section~\ref{sec:semiformal}: graded algebras with equivalent filtrations on homogeneous components have isomorphic completions. Appendix~\ref{sec:local_Borel} recalls the local graded Borel lemma.

\section{$\mathbb{Z}$-graded algebras: local model}\label{sec:semiformal}

\noindent In this section, we discuss local models of free $\Z$-graded algebras. We focus on two filtrations on the symmetric algebra generated by nonzero-weight variables with coefficients in smooth functions on an open coordinate chart. We investigate the relationship between these filtrations and their associated completions.

\smallskip\noindent Let $V$ a finite-dimensional $\mathbb{Z}$-graded real vector space,
\beq\label{eq:graded_vector_space}
V = \bigoplus_{i \in \mathbb{Z}} V_i.
\eeq

\begin{deff} \label{def-finite-gr}  
A graded vector space $V$ (resp. graded manifold $M$, to be defined later) is said to be
\begin{itemize}
\item \emph{of finite degree} if the maximal and minimal degrees of generating elements are bounded, so that the decomposition \eqref{eq:graded_vector_space} is finite in both directions.
\smallskip
\item \emph{of finite graded dimension} if $\dim(V_i) < \infty$ for all $i \in \mathbb{Z}$.
\smallskip
\item \emph{of finite dimension} if it is both of finite degree and finite graded dimension.
\end{itemize}
\end{deff}

\noindent
Define $\cE_i=A_0\otimes V_i$ for all $i\in\Z\setminus \{0\}$, where $A_0 = C^\infty(V_0)$, and $\cE=\bigoplus_{i}\cE_i$.
Let $A$ be the graded symmetric algebra generated by $\cE$ over $A_0$:

\[
A = \mathrm{Sym}_{A_0} \bigg( \bigoplus_{i \neq 0} \cE_i \bigg)
= T_{A_0} \bigg( \bigoplus_{i \neq 0} \cE_i \bigg) \Big/ \big\langle
v_1 \otimes v_2 - (-1)^{p(v_1)p(v_2)} v_2 \otimes v_1
\big\rangle,
\]

\noindent where $p(v) \in \mathbb{Z}_2$ captures weight modulo 2 parity. Thus $A = \bigoplus_{i \in \mathbb{Z}} A_i$ is supercommutative $\mathbb{Z}$-graded algebra, supporting two complete filtrations by graded ideals (for the notion of filtration, see Definitions \ref{app:deff_filtration} and \ref{app:def_graded_filtered} in the Appendix.).

\smallskip\noindent{\bf Filtration 1.}

\[
A = F^0 A \supset F^1 A \supset F^2 A \supset \cdots,
\]
via $F^p A = \langle \bigoplus_{i \geq p} A_i \rangle$, where the latter is the graded ideal generated by all homogeneous elements of weight $\ge p$. 

\smallskip\noindent{\bf Filtration 2.}

\[
A = {F_{pol}}^0 A \supset {F_{pol}}^1 A \supset {F_{pol}}^2 A \supset \cdots,
\]
where ${F_{pol}}^p A = \big(\cI\big)^p$ for $\cI=\langle \bigoplus_{i \ne 0} A_i \rangle$. Here, $\big(\cI\big)^p$ denotes the corresponding powers of the ideal.

\medskip\noindent Following the completion scheme for graded-filtered algebras from Proposition \ref{prop:graded-completion} (Appendix~\ref{sec:app}), we define the completions of $A$ with respect to both filtrations. More precisely, degreewise completion yields
\beq
\widehat{A}_i &=& \varprojlim_p A^p_i\,, \hspace{1.5mm}\mathrm{where}\hspace{1.5mm} A^p_i \colon = (A / F^p A)_i \\
\widehat{B}_i &=& \varprojlim_p B^p_i\,, \hspace{1.5mm}\mathrm{where}\hspace{1.5mm} B^p_i \colon = (A / {F_{pol}}^p A)_i
\eeq

\noindent The corresponding $\Z-$graded algebras are $\widehat{A}=\bigoplus_i \widehat{A}_i$ and $\widehat{B}=\bigoplus_i \widehat{B}_i$.

\begin{theorem}[Equivalence of completions]\label{thm:equiv_compl}\normalfont
     Provided $V$ is finite-dimensional, $\widehat{A}$ and $\widehat{B}$ are isomorphic as $\Z$-graded supercommutative algebras.
\end{theorem}
\noindent
\begin{proof} By Proposition~\ref{prop:graded-completion}, it suffices to show that the two filtrations are equivalent on each $A_r$, $r\in\Z$. 

\medskip\noindent
Let ${\xi}=\{\xi_i\}_{i=1}^n$ and ${\eta}=\{\eta_j\}_{j=1}^m$ be homogeneous bases of $V_+ = \bigoplus_{r>0} V_r$ and $V_- = \bigoplus_{r<0} V_r$ with weights $\bm{\alpha}=(\alpha_1,\ldots, \alpha_n)$ and $-\bm{\beta}=(-\beta_1, \ldots, -\beta_m)$, respectively. By construction, all integers $\alpha_i$ and $\beta_j$ are positive. Thus there exist positive integers
\[
\alpha_{\min} = \min\{\alpha_i\}, \quad 
\alpha_{\max} = \max\{\alpha_i\},
\]
\[
\beta_{\min} = \min\{\beta_j\}, \quad 
\beta_{\max} = \max\{\beta_j\},
\]
\[
\kappa = \min\{\alpha_{\min}, \beta_{\min}\}.
\]

\noindent A monomial $\xi_1^{p_1} \cdots \xi_n^{p_n} \eta_1^{q_1} \cdots \eta_m^{q_m}$ belongs to $A_r$ if and only if
\begin{equation}\label{eq:total_weight}
\sum_{i=1}^n \alpha_i p_i - \sum_{j=1}^m \beta_j q_j = r.
\end{equation} 

\noindent Notice that $F^k A_r$ and ${F_{pol}}^l A_r$ are generated by monomials satisfying 
$\sum_{i=1}^n \alpha_i p_i \ge k$ and $\sum_{i=1}^n p_i + \sum_{j=1}^m q_j \ge l$, respectively, 
provided that the total weight condition \eqref{eq:total_weight} holds. 

\noindent To prove filtration equivalence, it suffices to exhibit cofinal sequences $(l_k)$ and $(k_l)$ satisfying
\begin{equation}\label{eq:F1_to_F2}
\sum_{i=1}^n \alpha_i p_i \ge k \quad \Rightarrow \quad \sum_{i=1}^n p_i + \sum_{j=1}^m q_j \ge l_k
\end{equation}
and the symmetric condition for $(k_l)$: 
\begin{equation}\label{eq:F2_to_F1}
  \sum_{i=1}^n p_i + \sum_{j=1}^m q_j \ge l \quad \Rightarrow \sum_{i=1}^n \alpha_i p_i \ge k_l  
\end{equation}
The following choice works:
\[
l_k =\max\left\{
k\left(
\frac{1}{\alpha_{\max}}+\frac{1}{\beta_{\max}}
\right)-
\frac{r}{\beta_{\max}},0
\right\}
\,,\quad
k_l=
\max\left\{
\frac{1}{2}(r+\kappa l),0
\right\}.
\]
Evidently, $\lim_{k\to\infty} l_k = \lim_{l\to\infty} k_l = \infty$; therefore, the sequences $(l_k)$ and $(k_l)$ are cofinal, and thus the filtrations are equivalent.
 $\square$  
\end{proof}

\section{$\Z$-graded manifolds via semiformal homogeneity structure}\label{sec:formal_homogeneity}

\noindent In this section, we continue working with graded filtered algebras and develop a perspective on finite-dimensional $\Z$-graded spaces as semiformal homogeneity structures.

\smallskip

\noindent It is worth employing the topological approach to filtered algebras (see Remark \ref{app:top_remark} in the Appendix~\ref{sec:app}): they are viewed as topological spaces where the filtration ideals form a fundamental system of neighborhoods of zero. The completion of such an algebra is then its topological completion, obtained by adjoining limits of all Cauchy sequences. Recall that a Cauchy sequence in a filtered algebra \(A\) is a sequence \(\{a_i\}_{i\ge 0}\subset A\) such that for every \(p\) there exists \(k\) with \(a_i - a_j \in F^p A\) for all \(i,j\ge k\). Equivalent filtrations yield isomorphic completions in both the algebraic and topological senses. In particular, this means that every Cauchy sequence with respect to the topology of the first filtration converges in the completion with respect to the second, and vice versa. On the other hand, if the relation holds in only one direction, it induces a partial order rather than an equivalence. More precisely, we say that $F_1\le F_2$ if there exists a cofinal sequence $(l_k)$ such that $F_1^k A\subset F_2^{l_k}A$ for all $k\ge 0$. In this case, every Cauchy sequence with respect to $F_1$ ($F_1-$Cauchy sequence) converges in the completion $\hat{A}_{F_2}$ of $A$ with respect to $F_2$, but not conversely. Obviously, $F_1\sim F_2$ if and only if $F_1\le F_2$ and $F_2\le F_1$.

\smallskip\noindent Returning to the filtrations \(F\) and \({F_{pol}}\) considered earlier, we showed in Theorem~\ref{thm:equiv_compl} that they are equivalent when \(V\) is finite-dimensional. If \(V\) is infinite-dimensional, then only \({F_{pol}} \le F\) holds, as the sequence \((k_l)\) defined above applies equally to the finite-dimensional graded case. However, the reverse inequality does not hold in general, as shown by the following example.

\begin{example}[$F\le {F_{pol}}$ fails reversely for $\dim (V)=\infty$]\label{ex:connverge}\normalfont
   Let \(V\) be a graded vector space with a homogeneous basis consisting of vectors \(\xi_i\) and \(\eta_i\) for \(i \ge 0\), of weights \(i\) and \(-i\), respectively. Then the sequence \((f_k)\) with \(f_k (\xi,\eta)= \sum_{i=1}^j \xi_i \eta_i\) is \(F\)-Cauchy but not \({F_{pol}}\)-Cauchy; thus, it converges in \(\hat{A}_F\) but not in \(\hat{A}_{{F_{pol}}}\). In other words, the weight-$0$ infinite series $f (\xi,\eta) = \sum_{i>0} \xi_i \eta_i$ belongs to $\hat{A}_F$ but not to $\hat{A}_{{F_{pol}}}$.
\end{example}

\noindent 
An operator $\, \phi \colon A \to A \,$ is \emph{continuous in the $F$-filtered topology} if for every $\, p \ge 0 \,$ there exists $\, q \ge 0 \,$ such that $\, \phi(F^q A) \subset F^p A \,$ and the map $p \mapsto q$ is cofinal.
Clearly, any operator that shifts the filtration by a constant \(\delta\), i.e., satisfies \(\phi(F^pA)\subset F^{p+\delta}A\) for all \(p\ge 0\), is continuous. A continuous operator admits a unique extension to the completion (Theorem~\ref{app:thm_cofinal}). An example of a continuous operator is the Euler vector field, which in homogeneous coordinates $x_1, \ldots, x_d$ of weight $0$, $\xi_i$ of weight $\alpha_i$ ($i=1, \ldots, n$), and $\eta_j$ of weight $-\beta_j$ ($j=1, \ldots, m$) is given by\footnote{The Euler vector field retains the same form \eqref{eq:Euler} on spaces of finite graded dimension and remains continuous with respect to both filtrations, as it preserves each filtration.
}
\begin{equation}\label{eq:Euler}
\e = \sum_i \alpha_i \xi_i \partial_{\xi_i} - \sum_j \beta_j\eta_j \partial_{\eta_j}.
\end{equation}

\noindent For each open subset $U \subset V_0$, define $A_0(U) = C^\infty(U)$, $\cE_i (U) = A_0 (U) \otimes V_i$ for $i \neq 0$, and
\begin{equation}
A(U) = \mathrm{Sym}_{A_0(U)} \left( \bigoplus_{i \neq 0} \cE_i (U)\right).
\end{equation}

\noindent We obtain a presheaf of $\Z$-graded algebras over $V_0$ together with two filtrations defined as above. When $V$ is finite-dimensional, each quotient $A(U)/{F_{pol}}^p A(U)$ is a free $A_0(U)$-module of finite rank. Consider the formal completion of $A(U)$ with respect to ${F_{pol}}$ (of the whole algebra, not componentwise):
\[
A_{\mathrm{form}}(U) = \varprojlim_p \big(A(U)/{F_{pol}}^p A(U)\big).
\]

\noindent The next proposition shows that finite-dimensional graded geometry embeds naturally into the "usual" formal completion. The proof of the statements below is straightforward and relies on standard facts about formal power series in finitely many generators (which can always be taken homogeneous).

\begin{proposition}\label{prop:graded_inside_formal}\normalfont
    Assume that $V$ is finite-dimensional. Then we have:
\begin{enumerate}
\item $U \mapsto A_{\mathrm{form}}(U)$ is a sheaf of supercommutative algebras (as the projective limit of such sheaves).
\smallskip
\item The Euler vector field defined by \eqref{eq:Euler} is a derivation of the sheaf.
\smallskip
\item The componentwise completion of each $A(U)_r$ with respect to the filtration is a subsheaf of $A_{\mathrm{form}}(U)$ consisting of homogeneous elements of weight $r$ with respect to the Euler vector field ${\e}$, i.e., those $f$ satisfying ${\e}(f) = r f$.
\smallskip
\item Every derivation $\cv$ of $\hat{A}$ of weight $r$ extends to $A_{\mathrm{form}}(U)$ of the same weight with respect to ${\e}$, i.e., it satisfies $[{\e}, \cv] = r \cv$. Moreover, this equation characterizes all homogeneous derivations of the corresponding graded algebra $\hat{A}$ of weight $r$.
\end{enumerate}
\end{proposition}

\begin{rem}\label{rem:formal_homogeneity_philosophy}\normalfont
    Proposition~\ref{prop:graded_inside_formal} illustrates the core principle of graded geometry \cite{grab-hom, GrabowskaGrabowski2024, Voronov2002}: the graded algebra of functions embeds as a subalgebra of functions on a space equipped with an Euler vector field, where homogeneous elements of weight $r$ are the eigenvectors of the Euler vector field with eigenvalue $r$. The ambient space together with the Euler vector field is called a \emph{homogeneity manifold}, and the Euler vector field a \emph{homogeneity structure}. Except when all $V_i$ for $i\neq 0$ are odd, the algebra of functions on a homogeneity manifold is not graded in the usual sense, i.e., it is not the direct sum of its homogeneous components. While this idea typically appears in the smooth category, it applies equally well to the formal (or semiformal, as some directions remain smooth) setting.

\smallskip\noindent
The advantage is that, even though the sheaves of pure-weight homogeneous components exist (in both finite- and infinite-dimensional graded cases, see \cite{KS2024}), their direct sum is merely a presheaf in the non-odd case, as the sheaf gluing axiom fails. Although this is minor (since we work with graded morphisms and derivations, for which the individual homogeneous sheaves suffice), embedding the graded algebra into a "genuine" sheaf of algebras provides a more geometric perspective.
\end{rem}

\section{Graded Borel-Whitney-type theorem}\label{sec:global_Borel}

\noindent In previous sections, we reviewed the local model of $\Z$-graded manifolds. It is built from a $\Z$-graded vector space or, in an arbitrary coordinate chart $U$, from a $\Z^*$-graded trivial vector bundle over $U$ ($\Z^*$ omits $i=0$, as $V_0$ or $U\subset V_0$ serves as the local $0$-weight base). 
Since the algebra of local functions must include polynomials in non-zero weight variables, smooth functions of weight-$0$ variables, and be stable under arbitrary graded changes of coordinates (hence smooth functions of such polynomials), the minimal such algebra consists of formal series in non-zero weight coordinates with coefficients that are smooth functions of weight-$0$ coordinates \cite{KS2024}. The latter (in finite dimensions) follows from Proposition~\ref{prop:graded-series}. 
In Section~\ref{sec:semiformal}, Theorem~\ref{thm:equiv_compl} we proved that for the finite-dimensional case, the two possible filtrations yield isomorphic completions and thus isomorphic local models of $\Z$-graded manifolds.

\smallskip\noindent 
The global $\Z$-manifold is obtained by gluing local models. This was constructed in \cite{KS2024} for manifolds of finite graded dimension (finitely many local coordinates per integer weight). There, local functions of fixed weight were shown to form a sheaf. 
The full algebra of local functions (direct sum of homogeneous components) is only a presheaf, lacking gluing. However, this does not hinder defining the category of $\Z$-graded manifolds, as we primarily use sheaves of fixed-weight functions. 
Sheafification of the full algebra has a geometric solution in finite dimensions (Section~\ref{sec:formal_homogeneity}, especially Proposition~\ref{prop:graded_inside_formal} and Remark~\ref{rem:formal_homogeneity_philosophy}): every finite-dimensional $\Z$-graded manifold is a \emph{semiformal} homogeneity structure (in Kotov--Salnikov terminology\footnote{Semiformal means only part of the ``directions'' are formal, i.e.\ local functions are formal series only in a part of coordinates.}). Fixed-weight functions embed as eigenspaces of the Euler vector field. Functions on such semiformal manifolds form a true sheaf over the base, with the Euler field encoding the grading.

\smallskip\noindent In \cite{KS2024}, a \emph{Batchelor--Gawedzki-type theorem} establishes that every smooth $\Z$-graded manifold of finite graded dimension is non-canonically isomorphic to a canonical $\Z$-graded vector bundle over the same base, viewed as a semiformal $\Z$-graded manifold. Its structure sheaf arises from completing the sheaf of fiberwise polynomial functions. In finite dimensions, this means any smooth $\Z$-graded manifold is isomorphic to the formal neighborhood of the zero section in a smooth $\Z$-graded vector bundle, equipped with its (fiberwise) Euler vector field.

\smallskip\noindent The Batchelor--Gawedzki theorem admits a category-theoretic interpretation. Consider the category $\mathsf{B}_\Z$, whose objects are smooth finite-dimensional $\Z^*$-graded vector bundles (all fiber coordinates have non-zero weights) and morphisms are smooth maps preserving the Euler vector field (\emph{homogeneity maps}). These are not generally vector bundle morphisms, though they cover base maps, as the zero section is the zero locus of the Euler field.
With each such bundle, associate the formal neighborhood of its zero section, carrying the induced homogeneity structure (a $\Z$-graded manifold). This correspondence is functorial: homogeneity maps yield morphisms of $\Z$-graded manifolds. The theorem implies this functor is surjective on objects.

\smallskip\noindent
The next question is whether it is full. In other words: given a morphism of $\Z$-graded manifolds, does there exist a homogeneity map between the corresponding Batchelor--Gawedzki bundles inducing it? In this section, we answer this question affirmatively.

\smallskip\noindent Let $E \to M$ be a $\Z$-graded vector bundle over a smooth manifold $M$. Denote by $\wE$ the corresponding $\Z$-graded manifold, obtained as a neighborhood of the zero section. By construction, $\wE$ is a vector bundle with formal $\Z$-graded fibers (possibly including a homogeneous subbundle of weight $0$).

\begin{theorem}[Graded Borel-Whitney-type theorem]\label{thm:Borel-Whitney}\normalfont
    Let $E\to M$ and $E'\to M'$ be finite-dimensional $\Z^*$-graded vector bundles. For any morphism $\phi\colon \wE'\to \wE$ of $\Z$-graded manifolds, there exists a smooth homogeneity map $\widetilde{\phi}\colon E'\to E$ inducing $\phi$.
\end{theorem}

\noindent {\bf Proof (of Graded Borel-Whitney-type theorem).}

\smallskip\noindent For local $\Z$-graded manifolds, the answer follows immediately from the graded Borel lemma~\cite{KS2024} (see also Theorem~\ref{cor:gr-borel} in Appendix~\ref{sec:local_Borel}). 
Let $\cM$ and $\cM'$ be $\Z$-graded manifolds with homogeneous coordinates 
\[
(x, \theta) = (x_1, \ldots, x_{n_1}, \theta_1, \ldots, \theta_{n_2}), 
\quad
(z, \zeta) = (z_1, \ldots, z_{m_1}, \zeta_1, \ldots, \zeta_{m_2}),
\]
where all $x_i$ and $z_j$ have weight $0$, while each $\theta_a$ and $\zeta_b$ has a nonzero weight. 
Assume that the weights of the coordinates $\theta_a$ are denoted by $k_a$ for $a = 1, \ldots, n_2$.

\smallskip\noindent
A morphism $\phi \colon \cM' \to \cM$ is then determined by $n_1$ homogeneous functions $g_i$ of weight $0$ 
and $n_2$ homogeneous functions $h_a$ of weights $k_a$, respectively, which are smooth in the zero-weight variables and formal power series in the nonzero-weight ones:
\[
\phi^*(x_i) = g_i(z, \zeta), 
\quad 
\phi^*(\theta_a) = h_a(z, \zeta),
\qquad
i = 1, \ldots, n_1,
\quad
a = 1, \ldots, n_2.
\]
By Theorem~\ref{cor:gr-borel}, each function $g_i$ and $h_a$ admits a smooth lift of the same weight satisfying the required lifting property.

\smallskip\noindent The situation becomes only slightly more involved if we replace the flat source manifold with a non-flat one with a paracompact base $M'$, while keeping the target space flat. Indeed, in the same homogeneous coordinates $(x_i,\theta_a)$ on the target, the morphism $\phi$ is still given by functions $g_i$ and $h_a$ of appropriate weights. One applies Theorem~\ref{cor:gr-borel} on an open cover of $\cM' = \wE'$ trivializing $E' \to M'$ and glues the extensions together using a partition of unity subordinate to this cover\footnote{For a countable family of vector bundles over a paracompact smooth manifold, there exists a single open cover on which all bundles trivialize simultaneously, together with a partition of unity subordinate to this cover.}, yielding a global map $E' \to E$.

\smallskip\noindent If the target $\Z$-graded manifold $\wE$ is non-flat (meaning its base $M$ is non-flat), but $\phi\colon \wE' \to \wE$ is a bundle map (although not necessarily a vector bundle map), we can still apply the local theorem and gluing construction as before. We choose an open cover of $M$ and $M'$ that simultaneously trivializes both $E'$ and $E$.

\smallskip\noindent The general situation is more subtle.
To be on the safe side, recall that any morphism of $\Z^*$-graded manifolds induces a canonical base map via the canonical inclusions of the bases into the graded manifolds; every such morphism preserves these inclusions. Any morphism of $\Z$-graded manifolds is uniquely determined by its pullback $\phi^*$ on $\Z$-graded functions on the target space.
 Since the target is a vector bundle over $M$, it suffices to specify $\phi^*$ on smooth functions on $M$ and smooth sections of $E^*$ of fixed integer weight. 

 \smallskip\noindent
Taking into account the bundle structure of $\wE'$, we identify functions on $M$ with their pullbacks by the projection $\wE' \to M'$. This decomposes the space of functions on $\wE'$ into the direct sum of the subalgebra $C^\infty(M')$ and the ideal $\cI(\wE')$ of functions vanishing on the zero section (the latter consists of nilpotent elements with respect to the filtration ${F_{pol}}$; see Section~\ref{sec:semiformal}). Thus, for any function $f$ on $\wE$, one has
\begin{equation}\label{eq:base_plus_nilpotent}
  \phi^* f = \phi_0^* (f|_M) + (\phi^* f)_{\cI},
\end{equation}
where $f|_M$ denotes the restriction of $f$ to the zero section and $(\cdots)_{\cI}$ is the projection onto $\cI(\wE')$. In particular, if $f$ is induced by a homogeneous-weight section of $\wE^*$, then the first term in the decomposition \eqref{eq:base_plus_nilpotent} vanishes, since a morphism preserves the inclusion of bases (and thus the property of vanishing on them). On the other hand, for any base function $f$ on $M$, the nilpotent component of $\phi^*f$ is nonzero, unless $\phi$ is a bundle map.  

\smallskip\noindent 
To convey the idea, assume for the moment the existence of global coordinates $(x_i, \theta_a)$ on $\wE$, where all $x_i$, $i=1,\ldots,n_1$, have zero weight and all $\theta_a$, $a=1,\ldots,n_2$, have nonzero integer weights.
 Then $\phi^*\theta_a \in \cI(\wE')$, while
\[
\phi^* x^i = \phi_0^* x^i + u_i, \quad u_i = (\phi^* x_i)_{\cI}.
\]
One can represent the action of $\phi^*$ on functions $f(x,\theta)$ as the composition of the infinite jet prolongation $f \mapsto j^\infty(f)$, written in terms of its generating formal power series
\[
j^\infty(f)(x,\vx,\theta) = \sum_{(i_1,\ldots,i_{n_1})} {\pt_{x_1}^{i_1}} \cdots {\pt_{x_{n_1}}^{i_{n_1}}} f(x,\theta) \frac{\vx_1^{i_1}}{i_1!} \cdots \frac{\vx_{n_1}^{i_{n_1}}}{i_{n_1}!} ,
\]
viewed as a smooth function of $(x_i)$ with values in formal power series in $(\vx_i,\theta_a)$, and the morphism sending
\[
x_i \mapsto \phi_0^* x^i \in C^\infty(M'),\quad \vx_i \mapsto u_i \in \cI(\wE'),\quad \theta_a \mapsto \phi^*\theta_a
\]
for all $i=1,\ldots,n_1$ and $a=1,\ldots,n_2$. In this way, we reduce the general case to a morphism of formal bundles, though it still relies on target coordinates which may not exist globally.

\smallskip\noindent
The regular proof requires more refined techniques. We first prove the following lemma, then show how it completes Theorem~\ref{thm:Borel-Whitney}. As established, the notation $\hat{(\cdot)}$ for a (graded) vector bundle denotes the formal neighborhood of the zero section.


\begin{lemma}\label{lem:formal_geodesic}\normalfont
$\Z$-graded morphisms $\wE'\to\wE$ with the base map $\phi_0$ are in (noncanonical) one-to-one correspondence with homogeneity formal bundle maps $\wE'\to\wF$ over $\phi_0$, where $F=TM\times_M V$ is viewed as a $\Z$-graded vector bundle over $M$ with $TM$ being the homogeneous component of weight $0$.
\end{lemma}
\noindent
\begin{proof} In the proof of the lemma, we follow the approach of \cite{Bonavolonta:1304.0394}. The next step is to formulate an intrinsic version of the previous construction, expressing the pullback map $\phi^*$ for a morphism of $\mathbb{Z}$-graded manifolds as the composition of the infinite jet prolongation and the pullback along a homogeneity morphism of formal graded bundles.

\smallskip\noindent Consider the Cartesian product $M \times M$ together with the diagonal embedding 
\[\Delta: M \hookrightarrow M \times M.\] 
Denote by $M^{(\infty)} = M^{(\infty)}_{M\times M}$ the formal neighborhood of the diagonal in $M\times M$. Let $\p_1, \p_2$ be the projections of $M \times M$ onto the first and second factors, and $\wpr_1, \wpr_2$ their restrictions onto $M^{(\infty)} \subset M \times M$.

\smallskip\noindent
It is well known\footnote{The relation between jets and formal neighborhoods of the diagonal originates from algebraic geometry and differential geometry literature, prominently featured in Grothendieck's \emph{EGA} (\emph{\'{E}l\'ements de G\'eom\'etrie Alg\'ebrique}) \cite{grothendieck_ega} and developed by Malgrange for systems of PDEs \cite{malgrange_1962} (see also \cite{kock1980}).} that for any vector bundle $E\to M$, $\wpr_2^*E$ identifies with the bundle (along $\wpr_1$) of infinite jets of sections of $E$, and the pullback map $\wpr_2^*$ sends any section of $E$ to its infinite jet prolongation. In particular, $\wpr_2^*$ induces the infinite jet prolongation on smooth functions on $M$.

\smallskip\noindent Analogously, take the Cartesian product $M \times E$ with projections $\p_M\colon M \times E \to M$ and $\p_E\colon M \times E \to E$. Consider the embedding $M \hookrightarrow M \times E$ via the identity on $M$ and zero section of $E$, and its formal neighborhood $M^{(\infty)}_{M \times E} \subset M \times E$ with restricted projections $\wpr_M, \wpr_E$.
For the better visualization of the constructed morphisms and relations between them, one can use the following picture, where the implication arrow between the diagrams means that the right one is obtained from the left one by taking the formal neighborhood of $M$ embedded correspondingly:

\[
\xymatrix@R=1em@C=1.5em{
  & M\times E \ar[ddl]_{\p_M}\ar[dd] \ar[r]^-{\p_E} & E\ar[dd] &&   
  && M^{(\infty)}_{M \times E} \ar[ddl]_{\wpr_M}\ar[dd] \ar[r]^-{\wpr_E} & \wE\ar[dd]\\
  && &\ar@{=>}[r]& 
  &&& \\
  M & M\times M \ar[l]^-{\p_1} \ar[r]_-{\p_2}& M  &&
  &M & M^{(\infty)} \ar[l]^-{\wpr_1} \ar[r]_-{\wpr_2}& M
  } 
\]

\smallskip\noindent 
Consider the map
\[
\psi\colon \wE' \xrightarrow{\phi_0 \times \phi} M \times \wE.
\]
Here, for simplicity, we denote the composition of the projection $\wE' \to M'$ with the base map $\phi_0$ by the same letter $\phi_0$.

\smallskip\noindent Thanks to the embeddings $M \hookrightarrow M \times M \hookrightarrow M \times E$ and noting that $\wE = M^{(\infty)}_E$, which implies $M \times \wE = (M \times M)^{(\infty)}_{M \times E}$, we obtain an inclusion of formal neighborhoods 
\[
M^{(\infty)}_{M \times E} \hookrightarrow M \times \wE.
\]
It is not difficult to show (cf.~the super case in \cite{Bonavolonta:1304.0394}, whose arguments adapt directly) that $\psi$ takes values in $M^{(\infty)}_{M \times E}$. This induces a bundle morphism over $\phi_0$:
\[
\xymatrix@R=1em@C=1.5em{
 && \wE    && \\
 \wE'\ar[dr] \ar[urr]^-\phi \ar[rrrr]^-\psi
 &&&& M^{(\infty)}_{M \times E} \ar[ull]_-{\wpr_E} \ar[dl]  \\
 & M' \ar[rr]^-{\phi_0} && M &
 }
\]

\noindent Furthermore, any morphism $\phi\colon \wE' \to \wE$ corresponds bijectively to a bundle morphism $\psi\colon \wE' \to M^{(\infty)}_{M \times E}$ as above. The map $\phi \mapsto \psi$ was defined earlier; conversely, $\phi = \wpr_E \circ \psi$.

\smallskip\noindent
We now apply the exponential map technique to vector bundles over Riemannian manifolds. Every smooth manifold (compact or not) admits a \emph{complete Riemannian metric} \cite{nomizu_ozeki}.
On a complete Riemannian manifold $(M,\cg)$, geodesics with arbitrary initial velocity $\cv_x$ at $x\in M$ exist for all time. Thus the \emph{geodesic exponential map} \(\exp_{\cg}\colon T_xM\to M\) at every point $x\in M$ is defined on the entire tangent space $T_xM$ (\emph{Hopf-Rinow theorem}) \cite{hopf_rinow}. The geodesic map 
\[\exp_\cg\colon TM\to M\times M,\quad (x,\cv_x)\mapsto(x,\exp_{\cg}(x,\cv_x)),\] 
is a local diffeomorphism.
It does not yield a global diffeomorphism $TM\cong M\times M$ in general.
However, it identifies the formal neighborhood of the zero section in $TM$ with that of the diagonal in $M\times M$:
\[
\widehat{\exp_\cg}\colon \widehat{TM}\simeq M^{(\infty)}
\]
Choose a weight-preserving vector bundle connection on $E$ (a collection of connections on each $E_i$). Denote by $P^\nabla_{\gamma(t)}$ the parallel transport along the curve $[0,1]\mapsto \gamma(t)\in M$, with $\gamma(0)=x$ and $\gamma(1)=y$:
\[
P^\nabla_{\gamma(t)}\colon E_x\simeq E_y.
\]
Parallel transport along geodesics gives a linear isomorphism for all $(x,\cv_x)\in T_xM$,
\[
P^\nabla_{\exp(x,t\cv_x)}\colon E_x\simeq E_y, \quad y=\exp(x,\cv_x),
\]
or a local bundle isomorphism $F=TM\times_M E\to M\times E$ over $\exp_\cg\colon TM\to M\times M$; we denote this by $\Exp$.
Restricting to the formal neighborhood of the zero section in $F$ gives an isomorphism 
$\wF\simeq  M^{(\infty)}_{M \times E}$ over $\widehat{\exp_\cg}\colon \widehat{TM}\to M^{(\infty)}$ and therefore an isomorhism of $\Z-$graded  formal bundles over $M$. For a clearer illustration of the constructed morphisms and their interrelations, one can employ a composite diagram once more, where the implication arrow between the diagrams again signifies that the right one derives from the left via the formal neighborhood of $M$ embedded accordingly:

\[
\xymatrix@R=1em@C=1.5em{
  F \ar[rr]^-\Exp \ar[dd] \ar[dr] && M\times E \ar[dd]  \ar[dl]_-{\p_M} 
  && 
  &\wF \ar[rr]^-{\widehat{\Exp}} \ar[dd] \ar[dr] && M^{(\infty)}_{M \times E} \ar[dd]
  \ar[dl]_-{\wpr_M}\\
  & M & &\ar@{=>}[r]& 
  & & M\\
  TM \ar[rr]_-{\exp_\cg} \ar[ur]   & & M \times M \ar[ul]_-{\p_1}
  &&
  & \widehat{TM} \ar[rr]_{\widehat{\exp_\cg}}\ar[ur]     && M^{(\infty)}
  \ar[ul]_-{\wpr_1}
  } 
\]

\smallskip\noindent
Combined with the prior correspondence between morphisms $\wE'\to\wE$ and homogeneity bundle maps $\wE'\to M^{(\infty)}_{M \times E}$, this completes the proof of lemma.
    $\square$
\end{proof}

\medskip
\noindent Now we apply Lemma~\ref{lem:formal_geodesic} to finalize  the proof of Theorem~\ref{thm:Borel-Whitney}.
Choose compatible open covers of $M$ and $M'$ that trivialize $\phi_0^*F$ over $M'$. By the local graded Borel theorem~\ref{cor:gr-borel} and using a partition of unity on $M'$, we construct a homogeneity bundle map $E'\to F$ extending $\phi$. Composing this map with $\Exp\colon F\to M\times E=\p_2^*E$ and the projection $\p_2^*E\to E$ over $\p_2$ then yields the desired result.
    $\square$

\section{Beyond the Present Scope: Homogeneity structures}\label{sec:beyond}

\noindent In the previous sections of the article, we considered smooth graded vector bundles equipped with a vertical Euler vector field that defines the grading as an example of homogeneity manifolds. The latter are defined as smooth supermanifolds with a vector field which, in a neighborhood of the zero locus, takes the form of an Euler vector field in special homogeneous coordinates. Homogeneity manifolds, together with homogeneity (Euler-field-preserving) maps, form a category, which we denote by $\mathsf{HMan}$.
 There are other substantial classes of examples, which we present below.

\begin{example}[Homogeneity Lie groups \cite{KS-HCP}]\label{ex:homogeneity_group}\normalfont
    A $\mathbb{Z}$-graded Lie supergroup is a group object in the category 
$\mathsf{HMan}$; that is, a Lie supergroup endowed with a multiplicative 
homogeneity structure. Such a group realizes the third Lie theorem 
for $\mathbb{Z}$-graded Lie algebras. 
Differentiation at the identity of the multiplicative Euler vector field 
induces a compatible grading on the associated Lie (super)algebra 
$\mathfrak{g}$. Conversely, any compatible $\mathbb{Z}$-grading on 
$\mathfrak{g}$ integrates to a simply connected Lie supergroup 
carrying a multiplicative Euler field. 
The zero locus of the homogeneity structure on $G$ is the Lie subsupergroup $G_0$, whose Lie algebra $\mathfrak{g}_0 \subset \mathfrak{g}$ consists of 
elements of degree $0$. The exponential map preserves the homogeneity 
structure, and the Euler vector field on $G$ is linearizable in a neighborhood of $G_0$, with weights coinciding with those of $\mathfrak{g}$. The formal neighborhood of $G_0$, regarded as a group object in the category 
$\mathsf{Man}_{\mathbb{Z}}$, can be constructed by the method of Harish-Chandra pairs 
applied to the pair $(G_0, \mathfrak{g})$. 

\smallskip\noindent
An example of a homogeneity Lie group is $G = SL(2, \mathbb{R})$ equipped with 
the multiplicative Euler vector field described below. Let $h, e, f$ be the standard 
basis of $\mathfrak{g} = \mathfrak{sl}(2, \mathbb{R})$, satisfying 
$[h, e] = 2e$, $[h, f] = -2f$, and $[e, f] = h$. The adjoint action of $h$ on 
$\mathfrak{g}$ induces a $\mathbb{Z}$-grading with weights $(-2, 0, 2)$. 
The corresponding multiplicative vector field on $G$ is given by 
\(\epsilon_g = hg - gh\) for all \(g \in G\). 
The zero locus \(G_0\) of \(\epsilon\) consists of diagonal matrices 
with determinant \(1\). Although \(G\) itself is connected, \(G_0\) is disconnected, 
having two components containing \(Id\) and \(-Id\). The weights of \(\epsilon\) 
in a neighborhood of \(G_0\) coincide with those of the grading on 
\(\mathfrak{g}\), namely \( (-2, 0, 2)\).
\end{example}

\begin{example}[Projective spaces]\normalfont
  On $\mathbb{P}^1$ with homogeneous coordinates $[x_0:x_1]$, the vector field $\e = x_1 \partial_{x_1} - x_0 \partial_{x_0}$ defines a homogeneity structure, vanishing at the ``zero'' point ($x_1=0$) and at ``infinity'' ($x_0=0$). Locally, in the chart $z = x_1/x_0$ it becomes $\e = 2z \partial_z$, while in $w = x_0/x_1$ we obtain $\e = -2w \partial_w$. This construction extends analogously to higher-dimensional projective spaces $\mathbb{P}^n$. This example is closely related to Example~\ref{ex:homogeneity_group}, 
as it can be regarded as a homogeneous space of a homogeneity Lie group. 
The zero locus of the associated Euler vector field is disconnected 
and consists of two points. Although the zero locus of a multiplicative Euler field 
on a Lie supergroup can also be disconnected, this example exhibits 
a new phenomenon: the Euler field on the projective space has different 
weights in the neighborhoods of the distinct connected components.
\end{example} 

\smallskip\noindent Homogeneity structures provide a relatively new framework for studying gradings on 
manifolds, and several open questions remain whose answers may lead to interesting 
mathematical results. One such question, related to the dynamics of Euler vector fields, 
asks under what conditions a homogeneity structure $(M, \epsilon)$ is 
diffeomorphic, via a homogeneity diffeomorphism, to a graded vector bundle 
over the zero locus of $\epsilon$. This problem is probably connected to the study of homogeneous Riemannian metrics and affine connections on $M$. 

\smallskip\noindent Another question related to the topic of this paper concerns the development of an appropriate differentiable calculus on a $\mathbb{Z}$-graded vector bundle with a countable number of homogeneous finite-rank components. As is known from the Batchelor-Gawedzki-type theorem, every $\mathbb{Z}$-graded manifold of finite graded dimension is isomorphic to the $\mathbb{Z}$-graded manifold associated with such a graded vector bundle. For a "good" differentiable calculus, the graded Borel theorem (on the extension of any homogeneous semiformal function to a smooth function of the same weight) should hold.

\smallskip\noindent
It would also be interesting to investigate a Borel-Whitney-type theorem on the extension of morphisms for more general homogeneity structures, possibly with many connected components in the zero locus of the homogeneity structure (i.e., the Euler vector field).



\appendix

\section{Filtrations and completions 
}\label{sec:app}

\setcounter{equation}{0}
\renewcommand{\theequation}{\thesection.\arabic{equation}}

\noindent \noindent All \emph{material} in this section, except Definition~\ref{app:def_graded_filtered} and Proposition~\ref{prop:graded-completion}, 
is standard; see e.g. \cite{may-filtrations, stacks-project}.

\begin{deff}[Projective limit]\label{app:deff_projlim}
A \emph{projective system} of groups (rings, modules, etc.) indexed by a poset $I$ consists of 
objects $A^i$ ($i\in I$) with transition maps $\pi_{ij}\colon A^j\to A^i$ ($i\le j$) satisfying
\begin{equation}\label{app:eq_inverse_compatibility}
   \xymatrix{
A^k \ar[dr]_{\pi_{ik}} \ar[r]^{\pi_{jk}} & A^j \ar[d]^{\pi_{ij}} \\
& A^i
} 
\end{equation}
for $i\le j\le k$. The \emph{projective limit} is $\hat{A} = \varprojlim A^i$ with projections 
$\pi_i\colon\hat{A}\to A^i$ compatible with the $\pi_{ij}$.
\end{deff}

\noindent The explicit formula for the projective limit is
\begin{equation}\label{app:eq_proj_limit}
  \hat{A} = \bigl\{ \vec{a}\in\prod_{i\in I} A^i \bigm| \pi_{ij}(a_j)=a_i \ (i\le j) \bigr\}. 
\end{equation}

\begin{deff}\label{def:cofinal}\normalfont
A subset $J \subset I$ is called \emph{cofinal} in the poset $I$ if for every $i \in I$, there exists $j \in J$ such that $j \geq i$.
A monotonically increasing map $f \colon K \to I$ between posets is called \emph{cofinal} if its image $f(K)$ is cofinal in $I$.
\end{deff}

\begin{theorem}[Cofinal subsystem]\label{app:thm_cofinal}\normalfont{\mbox{}\vskip 2mm}
\begin{enumerate}
    \item Let $(A^i,\pi_{ij})_{i\in I}$ be a projective system over a poset $I$, and 
$J$ be \emph{cofinal} in $I$. 
Then \(
\varprojlim_{i\in I} A^i = \varprojlim_{j\in J} A^j
\).
\item Compatible morphisms $\phi_k\colon A^k\to B^{f(k)}$ over cofinal $f\colon K\to I$ extend uniquely to 
$\hat{\phi}\colon\hat{A}\to\hat{B}$ respecting projections.
\end{enumerate}
\end{theorem}

\noindent The proof relies on the fact that each $a_i$ in \eqref{app:eq_proj_limit} is uniquely determined by any $a_j$ with $j\ge i$ via the relation $a_i=\pi_{ij}(a_j)$. 
The compatibility \eqref{app:eq_inverse_compatibility} ensures the inverse limit is well-defined and independent of choices along cofinal subsets.

\smallskip\noindent From now on, we assume every poset $I$ is \emph{pointed}, i.e., it has a \emph{bottom element} $0=0_I$ satisfying $0_I\le i$ for all $i\in I$.

\begin{deff}\label{app:deff_filtration}\normalfont
A \emph{decreasing filtration} $F$ on a module $A$ indexed by the poset $I$ is a family of submodules $\{F^i A\}_{i\in I}$ such that \( A^j \subseteq A^i\) whenever $i\le j$. 
A decreasing filtration $F$ is 
\emph{exhaustive} if $A = F^0 A$, \emph{separated} if $\bigcap_{i\in I} F^i A = 0$, 
and \emph{complete} if it is both exhaustive and separated. 

\smallskip\noindent
The completion of a module $A$ with respect to a decreasing filtration $\{F^i A\}_{i\in I}$ is the inverse limit
\[
\widehat{A} =\widehat{A}_F = \varprojlim_i A/F^i A,
\]
where $\{A/F^i A\}_{i\in I}$ forms an inverse system with canonical projection maps 
\[
\pi_{ij}\colon A/F^j A \to A/F^i A
\]
for all $i \leq j$.
\end{deff}

\begin{deff}[Equivalent filtrations]\label{app:def_equiv}\normalfont
Let $\{F_1^i A\}_{i \in I}$, $\{F_2^k A\}_{k \in K}$ be two decreasing (exhaustive, separated) filtrations. We say that $F_1$ is \emph{finer} than $F_2$, written $F_1 \leq F_2$, if there exists a cofinal map $i\mto k(i)$ such that $F_1^i A \subset F_2^{k(i)} A$ for all $i \in I$. Filtrations are \emph{equivalent} if $F_1 \leq F_2$ and $F_2 \leq F_1$.
\end{deff}

\noindent The defined equivalence relation is reflexive, symmetric, and transitive. To check transitivity, it suffices to consider the composition of the corresponding cofinal maps (which is again cofinal).

\begin{deff}\label{app:def_filtered_algebra}\normalfont
    An associative $R$-algebra $A$, where $R$ is a commutative ring, is called \emph{filtered by a poset $I$} if it admits a decreasing filtration $(F_i)_{i\in I}$ as an $R$-module such that each $F_i$ is a two-sided ideal of $A$.
\end{deff}

\noindent Since each $F^i A$ is a two-sided ideal for all $i\in I$, the quotient $A^i \coloneqq A/F^i A$ is an $R$-algebra and all projection maps $\pi_{ij}$ are $R$-algebra morphisms. Then $\hat{A}_F$ is an $R$-algebra, where multiplication is defined componentwise via \eqref{app:eq_proj_limit}.

\begin{lemma}[Completion isomorphism]\label{app:lem_iso_compl}\normalfont
Equivalent filtrations $F_1\sim F_2$ on $R-$module $A$ yield isomorphic completions $\hat{A}_{F_1}\cong\hat{A}_{F_2}$. If $A$ is a filtered $R-$algebra, then this is also an isomorphism of algebras. 
\end{lemma}

\noindent\begin{proof}
A consequence of Theorem~\ref{app:thm_cofinal}. Assume $F_1$ and $F_2$ are complete filtrations indexed by posets $I$ and $K$, respectively, and let $f\colon I\to K$, $g\colon K\to I$ be cofinal maps satisfying
\[
F_1^i A \subset F_2^{f(i)} A, \quad F_2^k A \subset F_1^{g(k)} A \quad \text{for all } i\in I,\ k\in K.
\]
This induces compatible $R$-module (or $R$-algebra) morphisms
\[
\phi_i = \pi_{f(i)i} \colon A/F_1^i A \to A/F_2^{f(i)} A, \quad \psi_k = \pi_{g(k)k} \colon A/F_2^k A \to A/F_1^{g(k)} A
\]
for all $i\in I$, $k\in K$, which by Theorem~\ref{app:thm_cofinal} extend to $R$-module (or $R$-algebra) morphisms $\phi\colon \hat{A}_{F_1}\to \hat{A}_{F_2}$ and $\psi\colon \hat{A}_{F_2}\to \hat{A}_{F_1}$. It is readily verified that $\phi$ and $\psi$ are mutual inverses. $\square$
\end{proof}

\begin{rem}\normalfont{\mbox{}\vskip 2mm}\label{app:rem_completion_filtration}
\begin{itemize}
    \item The completion $\hat{A}=\hat{A}_F$ carries a canonical filtration
    \( \left(F^i\hat{A}\right)_{i\in I}\)
    induced componentwise from that on $A$ via the explicit description \eqref{app:eq_proj_limit} of the inverse limit. This filtration shares the same properties as the filtration \(F\) on \(A\): it is separated (exhaustive, complete) whenever \(F\) is.
    \item To be more precise:
    \begin{equation*}
        F^i\hat{A}=\{\vec{a}\in\hat{A}=\varprojlim_{j\in I} A/F^j A\,\mid\, a_j\in F^iA \bmod
        \left(F^iA \cap F^j A\right)\}.
    \end{equation*}
    \item There is a canonical morphism of filtered $R$-modules ($R$-algebras) $A\to\hat{A}$ given by
\[
a\mapsto(\bar{a}_i)_{i\in I}\in\hat{A}, \quad\text{where}\quad\bar{a}_i=a\bmod F^iA.
\]
When $F$ is separated, this morphism is injective: \(A\hookrightarrow  \hat{A} \).
\item When the poset \(I\) is \emph{directed}, that is, for any \(i,j \in I\) there exists \(k \in I\) with \(i \le k\) and \(j \le k\), this filtered morphism induces an isomorphism
\[
A \bmod F^i A \simeq \hat{A} \bmod F^i \hat{A}
\]
for each \(i \in I\).
Indeed,
\[
\hat{A} \bmod F^i \hat{A} = \Bigl\{ (b_j)_{j \in I} \;\Bigm|\; b_j \in A \bmod (F^i A + F^j A),\ \pi_{jk}(b_k) = b_j\ \forall\, k \ge j \Bigr\}.
\]
If \(I\) is directed, each \(b_j\) is uniquely determined by \(b_i\): there exists an upper bound \(k\) for the pair \((i,j)\), so
\[
b_j = \pi_{jk}(b_k) = \pi_{jk}(b_i),
\]
since \(F^i A = F^i A + F^k A\) and hence \(A \bmod (F^i A + F^k A) \simeq A \bmod F^i A\).
As a corollary, the completion of $A$ with respect to $F\hat{A}$ is canonically isomorphic to $\hat{A}$ as a filtered $R$-module ($R$-algebra).
\end{itemize}
   \end{rem}

\begin{lemma}\label{app:lem_morphism_completion}\normalfont
Let $(A,F_1)$ and $(B,F_2)$ be filtered $R$-modules (or $R$-algebras) indexed by posets $I$ and $K$, respectively, and let $\phi\colon A\to B$ be a morphism respecting the filtrations in the sense that there exists a cofinal map $f\colon I\to K$ satisfying
\[
\phi\left(F_1^i A\right) \subset F_2^{f(i)} B \quad \text{for all } i\in I.
\]
Then $\phi$ extends canonically to a morphism $\hat{\phi}\colon \hat{A}_{F_1}\to \hat{B}_{F_2}$ that respects the induced filtrations according to the same compatibility condition:
\[
\phi\left(F_1^i \hat{A}_{F_1}\right) \subset F_2^{f(i)} \hat{B}_{F_2} \quad \text{for all } i\in I.
\]
\end{lemma}
\noindent
\begin{proof}
Similar to the proof of Lemma~\ref{app:lem_iso_compl}.
    $\square$
\end{proof}

\begin{rem}[Topological view]\label{app:top_remark}\normalfont
Equivalent filtrations on a module  induce the same \emph{filtration topology}, where the filtration submodules form a fundamental system of neighborhoods of zero. 
Consequently, their associated completions are canonically isomorphic as topological objects. Moreover, any morphism that is continuous with respect to one filtration (i.e., respects the filtration) is automatically continuous with respect to the equivalent filtration. 
By Lemma~\ref{app:lem_morphism_completion}, any filtration-compatible morphism lifts canonically to a morphism of completions. Moreover, this lifting is functorial with respect to composition.
\end{rem}

\smallskip\noindent   Henceforth, we assume that $I = \N_0=\N \cup \{0\}$ with the standard ordering. 

\begin{deff}\label{def:Cauchy}\normalfont
A sequence $(a_n)_{n \in \mathbb{N}_0}$ is called a \emph{Cauchy sequence} if for every $k \geq 0$ there exists $n(k) \in \mathbb{N}_0$ such that
\begin{equation}\label{eq:Cauchy_stability}
  a_{n(k)} \equiv a_{n(k)+1} \equiv a_{n(k)+2} \equiv \cdots \pmod{F^k A}.  
\end{equation}
Two sequences $(a_n)$ and $(b_m)$ are called \emph{equivalent} if for every $k \geq 0$ there exists $r(k) \in \mathbb{N}$ such that
\begin{equation}\label{eq:seq_equiv}
  a_{r(k)} - b_{r(k)} \equiv a_{r(k)+1} - b_{r(k)+1} \equiv a_{r(k)+2} - b_{r(k)+2} \equiv \cdots \pmod{F^k A}.  
\end{equation}
In both cases, $n(k)$ and $r(k)$ are taken to be minimal for each $k \geq 0$.
\end{deff}

\noindent The proof of the following proposition follows standard arguments and can be easily verified by the reader or located in standard textbooks on filtrations.

\begin{proposition}\label{prop:topology}\normalfont
    The next statements hold:
\begin{enumerate}
    \item $F_1\le F_2$ if and only if for each $l$ there exists $k$ such that $F_1^k A \subset F_2^l A$;
    \item a morphism $\phi \colon A \to B$ of filtered modules (rings, algebras) is \emph{continuous} in the filtration topology, i.e., for each $l \ge 0$ there exists $k$ such that $\phi(F^k A) \subset F^l B$, if and only if it is compatible with the filtrations, i.e., there exists a cofinal sequence $(l_k)_{k \in \mathbb{N}_0}$ such that $\phi(F^k A) \subset F^{l_k} B$ (cf.\ Theorem~\ref{app:thm_cofinal});
    \item if $F_1 \leq F_2$ then a Cauchy sequence with respect to $F_1$ is also a Cauchy sequence with respect to $F_2$;
    \item the sequential completion of $A$ is isomorphic to $\hat{A}$.
\end{enumerate}
\end{proposition}

\begin{rem}[Sequential completeness]\label{app:rem_sequential_completeness}\normalfont
 While completion via Cauchy sequences is classical for metric spaces, the filtration topology provides another setting where this sequential completion is well-defined. 

\noindent As shown in Proposition~\ref{prop:topology}, there is a bijection between equivalence classes of Cauchy sequences in $A$ and elements of its completion $\hat{A}_F$. 

\end{rem}

\begin{deff}[Graded filtered algebra]\label{app:def_graded_filtered}
A $\Gamma$-graded algebra $A=\bigoplus_{\gamma\in\Gamma}A_\gamma$, $A_{\gamma_1}A_{\gamma_2}\subset A_{\gamma_1+\gamma_2}$, is called \emph{graded-filtered} if it is filtered in the sense of Definition~\ref{app:def_filtered_algebra} and each $F^iA$ is a \emph{graded ideal},
\[
F^iA=\bigoplus_{\gamma\in\Gamma}F_iA_\gamma, \quad F^iA_\gamma=(F^iA)\cap A_\gamma.
\]
\end{deff}

\begin{proposition}[Graded completions]\label{prop:graded-completion}\normalfont
A $\mathbb{Z}$-graded filtered algebra $A$ admits a \emph{graded completion} with respect to the graded filtration, given componentwise by $\widehat{A} = \bigoplus_{\gamma\in\Gamma} \widehat{A}_\gamma$, where $\widehat{A_\gamma}=\varprojlim_i A_\gamma/(F^i A_\gamma)$. Equivalent filtrations (on each $A_\gamma$ independently) yield isomorphic graded completions.
\end{proposition}

\noindent\begin{proof}
In a $\Gamma$-graded algebra $A$, the homogeneous multiplications $A_{\gamma_1}\times A_{\gamma_2}\to A_{\gamma_1+\gamma_2}$ are continuous in the filtration topology and thus extend to $\widehat{A}_{\gamma_1}\times\widehat{A}_{\gamma_2}\to\widehat{A}_{\gamma_1+\gamma_2}$.

\noindent
Indeed, let $(a_n)_{n\in\mathbb{N}}$, $(a'_n)_{n\in\mathbb{N}}$ be Cauchy sequences in $A_{\gamma_1}$, $A_{\gamma_2}$. For each $k$, there exist $n(k)$, $n'(k)$ such that
\[
a_n\equiv a_{n(k)}\pmod{F_kA_{\gamma_1}}\ \forall n\ge n(k), \quad a'_n\equiv a'_{n'(k)}\pmod{F_kA_{\gamma_2}}\ \forall n\ge n'(k).
\]
Their product $(a_na'_n)_n$ is Cauchy:
\[
a_na'_n\equiv a_{n(k)}a'_{n'(k)}\pmod{F_kA_{\gamma_1+\gamma_2}}\ \forall n\ge\max(n(k),n'(k)).
\]
Moreover, this respects Cauchy equivalence classes, yielding a well-defined $\Gamma$-graded completion of $A$.  $\square$
\end{proof}


\section{Graded Borel theorem: local approach}\label{sec:local_Borel}

\noindent In this section, we review the local graded Borel theorem as presented in \cite{KS2024}.

\smallskip\noindent
Consider an open set $U\subset V_0$ and the graded algebra $\hat{A}(U)$ of formal power series over $U$ in the notations of Section~\ref{sec:semiformal} and Section~\ref{sec:formal_homogeneity}. Let $\xi=\{\xi_1, \ldots, \xi_n\}$ and $\eta=\{\eta_1, \ldots, \eta_m\}$  be homogeneous coordinates of weights $\bm{\alpha}=(\alpha_1,\ldots,\alpha_n)$ and $-\bm{\beta}=-(\beta_1,\ldots,\beta_m)$, respectively (all integers $\alpha_i,\beta_j>0$). 

\begin{proposition}\label{prop:graded-series}\normalfont
Every $f\in\hat{A}(U)$ is a polynomial in ${\xi},{\eta}$ with coefficients that are formal power series in finitely many weight-$0$ monomials.
\end{proposition}

\noindent\begin{proof}
Consider the general form~\eqref{eq:series} of a function in $\hat{A} (U)$:
\begin{equation}\label{eq:series}
f (x, {\xi}, {\eta})= \sum_{\bm{p},\bm{q}} f_{\bm{p}\bm{q}}(x) {\xi}^{\bm{p}} {\eta}^{\bm{q}},
\end{equation}
where $\bm{p}=(p_1,\ldots,p_n)$ and $\bm{q}=(q_1,\ldots,q_m)$ are multi-indices. For simplicity, omit odd variables (they yield no infinite sums). The goal is to rearrange terms so infinite series appear only in degree-zero parts. Homogeneity of weight $r$ requires $\bm{\alpha}\cdot\bm{p}-\bm{\beta}\cdot\bm{q}=r$ (cf.~\eqref{eq:total_weight}).

\smallskip\noindent
This Diophantine problem is classical in semigroup theory~\cite{clifford1,clifford2}: the solution monoid is finitely generated\footnote{This reappears in computer algebra~\cite{clausen}.}. Let $S(\bm{\alpha},\bm{\beta},r)$ be nonnegative solutions to $\bm{\alpha}\cdot\bm{p}-\bm{\beta}\cdot\bm{q}=r$, and $S(\bm{\alpha},\bm{\beta})$ those for $r=0$. Let $M(\bm{\alpha},\bm{\beta},r)\subset S(\bm{\alpha},\bm{\beta},r)$ and $M(\bm{\alpha},\bm{\beta})\subset S(\bm{\alpha},\bm{\beta})\setminus\{0\}$ be minimal w.r.t.~$\N_{\ge0}^{n+m}$-order.

\begin{lemma}[Finiteness properties]\label{lem:finiteness}{\mbox{}\vskip 2mm}
\begin{enumerate}
\item $M(\bm{\alpha},\bm{\beta},r)$ and $M(\bm{\alpha},\bm{\beta})$ are finite.
\item $S(\bm{\alpha},\bm{\beta})$ consists of $\N_{\ge 0}$-combinations of $M(\bm{\alpha},\bm{\beta})$.
\item $S(\bm{\alpha},\bm{\beta},r)=M(\bm{\alpha},\bm{\beta},r)+S(\bm{\alpha},\bm{\beta})$.
\end{enumerate}
\end{lemma}

\noindent Solutions thus decompose into finite particular plus finitely generated homogeneous parts. Assuming finite dimension, $f$ is a finite sum of weight-$r$ monomials parametrized by $M(\bm{\alpha},\bm{\beta},r)$ times weight-zero series in monomials parametrized by $M(\bm{\alpha},\bm{\beta})$.

\smallskip\noindent Explicitly, for $r\ne 0$,
\beq\label{eq:f-series}
f(x,{\xi},{\eta})=\sum_{(\bm{p},\bm{q})\in M(\bm{\alpha},\bm{\beta},r)} h_{\bm{p}\bm{q}}(x,z({\xi},{\eta})) {\xi}^{\bm{p}} {\eta}^{\bm{q}},
\eeq
where $h_{\bm{p}\bm{q}}(x,z)$ are smooth in $x$ and formal series in weight-zero $z_{\bm{p}\bm{q}}={\xi}^{\bm{p}}{\eta}^{\bm{q}}$, $(\bm{p},\bm{q})\in M(\bm{\alpha},\bm{\beta})$. This proves the claim. $\qedhere$
\end{proof}

\begin{theorem}[$\Z$-graded Borel lemma]\label{cor:gr-borel}\normalfont
Every homogeneous local formal power series is the Taylor expansion of a local smooth function of the same weight.
\end{theorem}
\noindent\begin{proof}
By classical Borel lemma, choose smooth $\tilde{h}_{\bm{p}\bm{q}}(x,z)$ with Taylor series $h_{\bm{p}\bm{q}}(x,z)$ in~\eqref{eq:f-series}. Set
\[
\tilde{h}(x,{\xi},{\eta}):=\sum_{(\bm{p},\bm{q})\in M(\bm{\alpha},\bm{\beta},r)} \tilde{h}_{\bm{p}\bm{q}}(x,z({\xi},{\eta})){\xi}^{\bm{p}}{\eta}^{\bm{q}}.
\]
$\qedhere$
\end{proof}

\section*{Acknowledgments}

\smallskip\noindent
The authors thank Alina Rogozna for her fundamental contribution to Appendix~\ref{sec:app}.
The research was supported by the grant ``Graded differential geometry with applications'' GA\u{C}R 24-10031K of the Czech Science Foundation.

\bibliography{BibGraded1}

\end{document}